\documentclass[11pt, epsfig]{article}
\usepackage{epsfig, amsmath, amssymb, amsthm, times}
\usepackage[T1]{fontenc}
\usepackage[utf8]{inputenc}
\usepackage[usenames,dvipsnames]{xcolor}

\newcommand{\bel}{\begin{equation}\label}
\newcommand{\ee}{\end{equation}}
      \newtheorem{theorem}{Theorem}[section]
       \newtheorem{proposition}[theorem]{Proposition}
       \newtheorem{corollary}[theorem]{Corollary}
       \newtheorem{lemma}[theorem]{Lemma}
       \newtheorem{remark}{Remark}[section]
\theoremstyle{definition}

\def\eps{\varepsilon}

\def\R{{\mathbb R}}

\def\N{{\mathbb N}}

\def\<{\langle}
\def\>{\rangle}
\def\P{\mathbb P}

\def\E{\mathbb E}

\def\eps{\epsilon}

\def\0{\underline 0}

\def\a{\underline a}

\def\1{\underline 1}

\def\cF{\mathcal F}
\def\be*{\begin{equation*}}
\def\ee*{\end{equation*}}
\def\bar*{\begin{eqnarray*}}
\def\ear*{\end{eqnarray*}}
\def\bel{\begin{equation}\label}
\def\ee{\end{equation}}

\title{From invariance under  binomial thinning\\ to unification of the Cauchy and the Go\l \k ab-Schinzel-type equations}
\author{Karol Baron\thanks{University of Silesia, Katowice, Poland; e-mail: karol.baron@us.edu.pl}\hspace{1mm} and Jacek Weso\l owski\thanks{Warsaw University of Technology, Warszawa, Poland; e-mail: wesolo@mini.pw.edu.pl}}

\begin{document}
\maketitle
\begin{abstract}
	We point out to a connection between a problem of invariance of power series families of probability distributions under binomial thinning and functional equations which generalize both the Cauchy and an additive form of the Go\l \k ab-Schinzel equation. We solve these equations in several settings with no or mild regularity assumptions imposed on unknown functions.
\end{abstract}


\section{Introduction: invariance under binomial thinning in power series families}
Functional equations we analyze in this paper arise naturally in an invariance problem involving Poisson-type random probability measures, known under the nickname {\em throwing stones and collecting bones}. The problem has two basic ingredients: the binomial thining operator and the power series family of probability distributions. 

\begin{enumerate}
	\item Binomial thinning: 
	
Let $\mathcal P(\N)$ be the set of probability measures with supports in $\N=\{0,1,\ldots\}$. For every $p\in[0,1]$ the binomial thining operator $T_p$ is defined as follows:
$$
\mathcal P(\N)\ni \mu\mapsto T_p(\mu)\in\mathcal P(\N)
$$
and $T_p(\mu)$ is the probability distribution of  
\bel{Ktil}
\tilde{K}:=\sum_{n=0}^K\,I_n,
\ee
where the sequence $(I_n)_{n\ge 1}$  of independent random variables with the same Bernoulli distribution $\mathrm{Ber}(p):=(1-p)\delta_0+p\delta_1$ (additionally we denote $I_0=0$) and the random variable $K$ with distribution $\mu$ (defined  on some probability space $(\Omega,\cF,\P)$)  are independent. Here and in the sequel by $\delta_x$ we denote the Dirac measure at $x$. In particular, $T_0(\mu)=\delta_0$  and $T_1(\mu)=\mu$.  

This operator was introduced in \cite{SvH} to establish discrete versions of stability and selfdecomposability of probability measures. Since then binomial thinning operator and its extensions have been intensively studied in various probabilistic contexts (a prominent example being the time series theory). In particular, very recently \cite{BR} (referred to by BR in the sequel) used the thinning operator to model Poisson-type random point processes restricted to a subset of the original state space. 

\item Power series family: 

Let $\a=(a_k)_{k\ge 0}$ be a sequence of nonnegative numbers with $a_0=1$  such that the set
$$
\Theta_{\a}=\left\{\theta\ge 0:\,\varphi(\theta):=\sum_{k\ge 0}\,a_k\theta^k<\infty\right\}
$$
has a non-empty interior (actually $\Theta_{\a}$ is a convex set). Then $$\mu_{\a,\theta}=\sum_{k\ge 0}\,\tfrac{a_k\theta^k}{\varphi(\theta)}\,\delta_k$$ is a probability measure  called a power series distribution generated by $\a$ with the parameter $\theta\in\Theta_{\a}$. The power series family generated by $\a$ is defined as $$\mathcal{PSF}(\a)=\{\mu_{\a,\theta}:\,\theta\in\Theta_{\a}\}.$$ 
\end{enumerate}

The problem lies in identification of  power series families which are  invariant under binomial thinning, i.e.  one searches for  $\mathcal{PSF}(\a)$ satisfying 
$$
T_p(\mathcal{PSF}(\a))\subset \mathcal{PSF}(\a)
$$
for some $p\in (0,1)$. 

Equivalently, we want to describe all sequences $\a$ of nonnegative numbers with $a_0=1$ and with  $\Theta_{\a}$ of non-empty interior, such that there exists $p\in(0,1)$ and a function $h_p:\Theta_{\a}\to\Theta_{\a}$ which satisfy the condition  
\bel{KTK}
\forall\,\theta\in\Theta_{\a}\quad\left(\, K\sim \mu_{\a,\theta}\quad \Rightarrow\quad \tilde{K}\sim\mu_{\a,h_p(\theta)}\,\right),
\ee
where $\tilde{K}$ is defined in \eqref{Ktil}. (If $\mu$ is the probability distribution of a random variable $X$ we write $X\sim \mu$.)

Probability generating function is a convenient tool to analyze this problem. Recall that the probability generating function $\psi_{\mu}$ of $\mu=\sum_{k\ge 0}\,p_k\delta_k\in \mathcal{P}(\N)$ is defined by
$\psi_{\mu}(s)=\sum_{k\ge 0}\,s^kp_k$ on a domain $U\supset[-1,1]$. In particular, $$\psi_{\mathrm{Ber}(p)}(s)=ps+q,\quad s\in\R,\quad \mbox{where }\;q=1-p,$$ and 
$$
\psi_{\mu_{\a,\theta}}(s)=\sum_{k\ge 0}\,s^k\,\tfrac{a_k\theta^k}{\varphi(\theta)}=\tfrac{\varphi(s\theta)}{\varphi(\theta)},\qquad |s|\theta\in\Theta_{\a},\;\theta\in\Theta_{\a},
$$
where $\varphi$ is defined on $\Theta_{\a}\cup(-\Theta_{\a})$ by analytical extension.

Let $\nu$ be the distribution of $\tilde{K}$ defined in \eqref{Ktil} with $K\sim\mu_{\a,\theta}$ for $\theta\in\Theta_{\a}$. Then using conditioning with respect to $K$ and independence of $K,I_1,I_2,\ldots$ we get
$$
\psi_{\nu}(s)=\E\,s^{\tilde{K}}=\E\,s^{\sum_{n=0}^K\,I_n}=\E\,(\psi_{\mathrm{Ber}(p)}(s))^K=\psi_{\mu_{\a,\theta}}\left(\psi_{\mathrm{Ber}(p)}(s)\right)=\tfrac{\varphi((ps+q)\theta)}{\varphi(\theta)},
$$
if only  $\theta,\,|ps+q|\theta\in \Theta_{\a}$.

By \eqref{KTK} we have $\nu=\mu_{\a,h_p(\theta)}$, i.e. for $p\in(0,1)$ we get the equation
\bel{row11}
\tfrac{\varphi((ps+q)\theta)}{\varphi(\theta)}=\tfrac{\varphi(sh_p(\theta))}{\varphi(h_p(\theta))}
\ee
for $s$ and $\theta$ satisfying  $\theta,\,|ps+q|\theta,\,|s|h_p(\theta)\in \Theta_{\a}$. Since $q\theta\in\Theta_{\a}$, upon inserting  $s=0$  in \eqref{row11},  we get (note that $\varphi(0)=1$)
$$
\varphi(h_p(\theta))=\tfrac{\varphi(\theta)}{\varphi(q\theta)},\quad \theta\in\Theta_{\a}.
$$
Consequently, \eqref{row11} can be rewritten as
\bel{basica}
\varphi((ps+q)\theta)=\varphi(q\theta)\,\varphi(sh_p(\theta)).
\ee
Then, upon changing variables  $u:=ps\theta$, $v:=q\theta$ the equation   \eqref{basica} yields 
\bel{proto}
\varphi(u+v)=\varphi(v)\,\varphi(u\rho(v))
\ee
on the proper domain for variables $u$ and $v$ (actually, this domain contains a  neighbourhood of zero for $u$ and a right neighbourhood of zero for $v$), where $\rho(0)=1$ and $\rho(v)=\tfrac{q}{pv}h_p\left(\tfrac{v}{q}\right)$, $v>0$. Since $\varphi(0)=1$ one can apply the logarithm to both sides of \eqref{proto} for $u$ in a two-sided neighbourhood of zero and $v$ in the right-neighbourhood of zero which leads to an additive version of this equation \eqref{proto}. 
Such equation, referred to by the {\em modified Cauchy equation} in BR, has been recently solved in that paper. We quote now this result {\em in extenso}:
\begin{lemma}[BR, Lemma 1]\label{BaRe}  Assume that $f(t)$ is twice differentiable in some neighbourhood of the origin, satisfies $f(0)=0$ and $f'(0)>0$ as well as
$$
f(s+t)-f(s)=f(h(s)t),
$$
where $h(s)$ is $t$ free. Then $f$ is of the form $f(t)=At$ or $f(t)=B\log(1+At)$ for some $A,B\neq 0$. Moreover $h(s)=f'(s)/f'(0)$.
\end{lemma}
In BR  Lemma \ref{BaRe} is used to identify  Poisson, binomial and negative binomial probability distributions as the only power series families which are invariant under binomial thinning (see Theorem 2 and its proof in BR).  

We are interested in the above equation as well as its "dual" $f(s+g(s)t)=f(s)+f(t)$. Instead of a neighbourhood of zero we consider domains: $[0,\infty)$ in Section 3 and $V$, a vector space, in Section 4. We assume minor or no regularity conditions on $f$ and consider several cases of image spaces of $f$: a unital magma, the real line and a linear topological space. No regularity conditions whatever are imposed on the unknown functions $h$ and $g$. Our results complete to some extent \cite{JCh} where all solutions of the "dual" equation are determined in the case when $f$ maps a linear space (real or complex) into a semigroup and $g$ satisfies some regularity conditions.  In Section 2 we give preliminaries on the functional equations we are interested in.

\section{Cauchy-Go\l \k ab-Schinzel equations}

Let $(M,+)$ be a magma, i.e.  $M$ is a set equipped with a binary operation $+:M\times M\to M$. Consider $U=[0,\infty)$ or $U=V$, a vector space over a field $\mathbb F$. For unknown functions $f:U\to M$ and $g,h:U\to W$, where $W=U$ in case $U=[0,\infty)$ and $W=\mathbb F$ in case $U=V$,  we consider  equations 
\bel{equ}
f(s+t)=f(s)+f(h(s)t),\qquad s,t\in U,
\ee
and
\bel{rew}
f(s+g(s)t)=f(s)+f(t),\qquad s,t\in U.
\ee
We would like to identify all solutions $(f,h)$ of \eqref{equ} and $(f,g)$ of \eqref{rew}. 

If $f=g$ (thus $M=W$), then \eqref{rew} becomes an additive form of the Go\l \k ab-Schinzel equation, see \cite{JA}, pp. 132-135, \cite{AD}, pp. 311-319 and the survey paper \cite{JB }. For more recent contributions on the Go\l \k ab-Schinzel equation and its generalizations consult e.g. \cite{CK14}, \cite{CK15}, \cite{CK17}, \cite{BO}  and \cite{O17}. In particular, the latter paper reveals yet another probabilistic (stable laws and random walks) connection  of the Go\l \k ab-Schinzel equation, treated there as a {\em disguised form} of the  Goldie equation. If $g\equiv 1$ or $h\equiv 1$, then  \eqref{equ} and \eqref{rew} become the same standard Cauchy equation.  Hence we call \eqref{rew} as well as \eqref{equ} the Cauchy-Go\l \k ab-Schinzel (CGS) equations. 

We consider unital magma $M$ (i.e. $M$ has a neutral element, denoted by $\bf 0$ throughout the paper) with  the two-sided cancelation property. To avoid trivialities we  assume that $f\not\equiv \bf 0$.

Note that for $f$ which solves either \eqref{equ} or \eqref{rew} we have
\bel{f0}
f(0)={\bf 0}.
\ee

\vspace{3mm}
\begin{remark}
	\label{rem1}
	If $(f,g)$ solves \eqref{rew}, then $\mathrm{Ker}(g):=\{s\in U:\,g(s)=0\}=\emptyset$.
	
	Assume not, i.e. $g(s_0)=0$ for some $s_0\in U$. Then \eqref{rew} implies $f(s_0)=f(s_0+g(s_0)t)=f(s_0)+f(t)$ for any $t\in U$. Hence $f\equiv \mathbf 0$, a contradiction.
\end{remark}

\begin{remark}\label{rem3}
	If  $(f,h)$ solves \eqref{equ} and $\mathrm{Ker}(h)=\emptyset$, then $(f,g)$ with $g=1/h$  solves \eqref{rew}. 
	In the opposite direction, if $(f,g)$    solves \eqref{rew}, then  $(f,h)$ with $h=1/g$ (being well-defined by Remark \ref{rem1}) solves \eqref{equ}.
\end{remark}

\begin{remark}\label{rem2}
Let $(f,h)$ solves \eqref{equ} for $U=V$, a vector space. Then $\mathrm{Ker}(h)=\emptyset$. 

Assume not,   i.e. $h(s_0)=0$ for some $s_0\in V$. Then \eqref{equ}  together with \eqref{f0} imply $f(s_0+t)=f(s_0)$ for every $t\in V$. That is, $f$ is a constant function. By \eqref{f0} we get a contradiction with $f\not\equiv {\bf 0}$.
\end{remark} 

When $U=[0,\infty)$, while considering \eqref{equ} it is convenient to distinguish two cases with respect to the form of the kernel of $h$: 
\begin{enumerate}
	\item[{\bf I}] $\mathrm{Ker}(h)\neq\emptyset$
	\item[{\bf II}] $\mathrm{Ker}(h)=\emptyset$
\end{enumerate}

\section{CGS equations on $U=[0,\infty)$}

\subsection{$\mathbf{\mathrm{\mathbf{Ker}}(h)\neq\emptyset}$}
Throughout this section we assume that $U=[0,\infty)$ and that  $(M,+)$ is a unital magma with the two-sided cancelation property.

\begin{theorem} 
Assume that $f:[0,\infty)\to M$ is non-zero, $h:[0,\infty)\to[0,\infty)$, $\mathrm{Ker}(h)\neq\emptyset$ and $$s_0=\inf\,\mathrm{Ker}(h).$$

Then $(f,h)$ solves \eqref{equ} on $[0,\infty)$ if and only if  
\begin{enumerate} \item 
	either $s_0=0$ and  
	\bel{fg1}
	f(s)=\left\{\,\begin{array}{ll} 
		\mathbf 0,& \mathrm{for}\;s=0, \\
		\mathbf a, & \mathrm{for}\;s\in(0,\infty),
	\end{array}\right. \qquad 	
	h(s)=\left\{\,\begin{array}{ll} 
		b,& \mathrm{for}\;s=0, \\
		0, & \mathrm{for}\;s\in(0,\infty),
	\end{array}\right.
	\ee
	where $\mathbf a\in M\setminus\{\mathbf 0\}$ and  $b\in(0,\infty)$,
	\item or  $s_0>0$  and
	\bel{fg2}
	f(s)=\left\{\,\begin{array}{ll} 
		\mathbf 0,& \mathrm{for}\;s\in[0,s_0), \\
		\mathbf a, & \mathrm{for}\;s\in[s_0,\infty),
	\end{array}\right. \qquad 	
	h(s)=\left\{\,\begin{array}{ll} 
		\tfrac{s_0}{s_0-s},& \mathrm{for}\;s\in[0,s_0), \\
		0, & \mathrm{for}\;s\in[s_0,\infty),	\end{array}\right.
	\ee
	where $\mathbf a\in M\setminus\{\mathbf 0\}$.
	\end{enumerate}
\end{theorem}	
	
	\begin{proof}
		It is easy to check that $(f,h)$ given in \eqref{fg1} and \eqref{fg2} solve \eqref{equ}.
		
		Let $(f,h)$ solve \eqref{equ}. Then it follows from \eqref{equ} and \eqref{f0} that for any $t\in[0,\infty)$,
		 $$
		 f(s+t)=f(s)\qquad\mbox{if only }\;h(s)=0.
		 $$
		 Since such $s\ge s_0$ can be chosen arbitrarily close to $s_0$ we conclude that 
		 \bel{(5)}
		 f(s)=\mathbf a\in M,\qquad s\in(s_0,\infty).
		 \ee
		 
		 If $s\in[0,\infty)$ is such that $h(s)\ne 0$, then we consider \eqref{equ} for $t>0$ such that $h(s)t>s_0$ and $s+t>s_0$. Then \eqref{(5)} yields $\mathbf a=f(s)+\mathbf a$,  whence $f(s)=0$. Therefore, 
		 \bel{imp}
		 h(s)\neq 0\quad \Rightarrow\quad f(s)=\mathbf 0,\qquad s\in[0,\infty).
		 \ee 
		 
		 Combining \eqref{(5)} and \eqref{imp} we get 
		 \bel{pmi}
		 h(s)=0,\quad s\in(s_0,\infty)\qquad\mbox{if only }\; \mathbf a\neq \mathbf 0.
		 \ee
		  
		Consider now two cases. 
		
		\begin{enumerate}
			\item $s_0=0$: Then \eqref{(5)} together with \eqref{f0} imply that $\mathbf a\neq \mathbf 0$. Consequently, $f$ is as given in \eqref{fg1}. 
			
			By \eqref{pmi} we have $h(s)=0$ for $s>0$. From \eqref{equ} for $s=0$ and $t>0$ we get $\mathbf a=f(t)=f(h(0)t)$. Thus \eqref{f0} implies $h(0)>0$.  Consequently,  $h$ is as given in \eqref{fg1}.
			
			\item $s_0>0$: Then \eqref{imp} gives $f(s)=\mathbf 0$ for $s\in[0,s_0)$. From \eqref{equ} for $s=s_0$ and $t>0$ such that $h(s_0)t<s_0$ by \eqref{(5)} we get $\mathbf a=f(s_0)$ which implies that $\mathbf a\neq \mathbf 0$. Consequently,  $f$ is as given in \eqref{fg2}. 
			
			By \eqref{pmi} and \eqref{imp} we have $h(s)=0$ for $s\in[s_0,\infty)$.  Let $s\in[0,s_0)$. Then $f(s+t)=f(h(s)t)$ for $t\in[0,\infty)$. Referring to $f$ as given in \eqref{fg2} we see that 
			$$
			s+t<s_0\quad \Leftrightarrow \quad h(s)t<s_0.
			$$
			Thus $$\tfrac{s_0}{s_0-s}\le h(s)<\tfrac{s_0}{s_0-s-\eps}\qquad \mbox{for}\;\;\eps\in(0,s_0-s].$$
			By taking $\eps\downarrow 0$ we obtain $h(s)=\tfrac{s_0}{s-s_0}$ for $s\in[0,s_0)$. Consequently, $h$ is as given in \eqref{fg2}.  
		\end{enumerate}

	\end{proof}
	
	\subsection{$\mathbf{\mathrm{\mathbf{Ker}}(h)=\emptyset}$}
	
	Let $U=[0,\infty)$. As explained in Remark \ref{rem3},  $(f,h)$ solves \eqref{equ} if and only if $(f,g)$ with $g=1/h$ solves \eqref{rew}. We start with the case where $f$ is injective; cf. \cite{EV}.
	Then we move on to the case when $f$ is right continuous at a point.
	
	\subsubsection{A magma version} 
	Throughout this section we assume that $(M,+)$ is a commutative unital magma with the (two-sided) cancelation property. 
	
	\begin{theorem}\label{inj}
		Assume that $f:[0,\infty)\to M$ is injective and $g:[0,\infty)\to [0,\infty)$.
		
		Then $(f,g)$ solves \eqref{rew} if and only if \begin{enumerate}
			\item  either $g(1)=1$ and
		\bel{fg3}
		f\;\mbox{is additive}\qquad\mbox{and}\qquad  g\equiv 1,
		\ee
		\item or $g(1)\neq 1$ and
		\bel{fg4}
		 f(s)=a(\log(\alpha s+1))\qquad\mbox{and}\qquad  g(s)=\alpha s+1,\quad s\in[0,\infty),
		 \ee
		 where $\alpha\in(0,\infty)$ and $a:[0,\infty)\to M$ is an injective additive function.
		 \end{enumerate}
\end{theorem}	 
		 \begin{proof} 
		 	It is easy to check that $(f,g)$ given in \eqref{fg3} and \eqref{fg4} solve \eqref{rew}. 
		 	
		 	By commutativity at the right hand side of \eqref{rew} and injectivity of $f$ we conclude that
		 	$$
		 	s+g(s)t=t+g(t)s,\quad s,t\in[0,\infty).
		 	$$
		 	For $t=1$ we get \bel{hs} g(s)=\alpha s+1,\quad s\ge 0,\ee
		 	where $\alpha=g(1)-1$ and $\alpha\ge 0$ since $g$ is positive.

		 	Consider two possible cases: $g(1)=1$ and $g(1)\neq 1$. 
		 	\begin{enumerate}
		 		\item $g(1)=1$: 
		 		
		 		By \eqref{hs} we have $\alpha=0$ and  $g\equiv 1$. Then by \eqref{rew} it follows that $f$ is additive.
		 		
		 		\item $g(1)\neq 1$:
		 		
		 		By \eqref{hs} we have $\alpha>0$ and $g$ is as given in \eqref{fg4}. 
		 		
		 	It follows from \eqref{hs} that $k:=\log(g):[0,\infty)\to[0,\infty)$ is a bijection (note that $g(s)\ge 1$ for all $s\in[0,\infty)$). Therefore the formula $a\circ k=f$ defines a function $a:[0,\infty)\to M$. Clearly, $a$ is injective and the equation \eqref{rew} in terms of $a$ assumes the form
		 		$$
		 		a(k(s))+a(k(t))=a(k(s+g(s)t)),\quad s,t\ge 0.
		 		$$
		 		But $g$, see \eqref{hs}, satisfies $g(s+g(s)t)=g(s)g(t)$, i.e. $k(s+g(s)t)=k(s)+k(t)$. Therefore  
		 		$$
		 		a(k(s))+a(k(t))=a(k(s)+k(t)),\quad s,t\ge 0.
		 		$$
		 		Since $k$ is bijective on $[0,\infty)$ we conclude that $a$ is an additive function. 
		 			\end{enumerate}
		 \end{proof}

	\subsubsection{Real and vector space versions}
	
	We first consider the case of real-valued $f$.
	
	\begin{theorem}\label{oner}
		Assume that $f:[0,\infty)\to \R$ is non-zero,  right continuous at some point and $g:[0,\infty)\to[0,\infty)$.
		
		Then $(f,g)$ solves \eqref{rew}  if and only if 
		\begin{enumerate}
			\item either
			$$
			f(s)=as,\quad s\in[0,\infty)\qquad\mbox{and}\qquad  g\equiv 1,
			$$
			where $0\neq a\in\R$,
			\item or
			$$
			f(s)=a\log(\alpha s+1)\qquad\mbox{and}\qquad g(s)=\alpha s+1,\quad s\in[0,\infty),
			$$
			where $\alpha>0$ and $0\neq a\in\R$.
		\end{enumerate}
	\end{theorem}
For the proof of Theorem \ref{oner} we use several auxiliary results which are considered first. 

In Lemmas \ref{rcrc}, \ref{rch}, \ref{rcc}, \ref{cvm} and \ref{mi} as well as in Proposition \ref{ccc} we assume that a non-zero function $f:[0,\infty)\to \R$ satisfies \eqref{rew} with some $g:[0,\infty)\to[0,\infty)$.

\begin{lemma}\label{rcrc}
	If $f$ is  right continuous at some point, then it is a right continuous function.
\end{lemma}

\begin{proof} 
	Let $f$ be right continuous at $s_1\in[0,\infty)$. From \eqref{rew} we have
	$$
	f(s_1+g(s_1)t)-f(s_1)=f(t),\quad t\ge 0.
	$$
	Taking $t\to 0^+$ we see that  right continuity of $f$ at $s_1$ implies that $f$ is right continuous at $0$. 
	
	Fix arbitrary $s>0$. Then by \eqref{rew} for $t\ge 0$ we have
	$$
	f(s+t)-f(s)=f(h(s)t),
	$$
	where $h=1/g$. Taking $t\to 0^+$ we conclude that $f$ is right continuous at $s$. 
\end{proof}

\begin{lemma}\label{rch}
If $f$ is right continuous, then $g(s)\ge 1$ for every $s\ge 0$.
\end{lemma}

\begin{proof} 
	Assume $g(s_1)<1$ for some $s_1\ge 0$. Then for $\phi:[0,\infty)\to[0,\infty)$ defined by $\phi(t)=g(s_1)t+s_1$,
	$$
	\tilde{s}<\phi(t)<t\qquad \forall\,t>\tfrac{s_1}{1-g(s_1)}=:\tilde{s}.
	$$
	Hence $t_n:=\phi^{\circ n}(t)=g^n(s_1)t+s_1\sum_{k=0}^{n-1}\,g^k(s_1)\to\,\tilde{s}^+$ as $n\to\infty$ for $t>\tilde s$. Thus right continuity of $f$ implies
	\bel{limf}
	f(\tilde{s})-f(t)=\lim_{n\to \infty}\,\left(f(t_n)-f(t)\right).
	\ee
	On the other hand for any $k\ge 1$ we have $t_k=\phi(t_{k-1})$ and thus \eqref{rew} yields
	$$
	f(t_k)=f(g(s_1)t_{k-1}+s_1)=f(t_{k-1})+f(s_1)
	$$
	whence
	$$
	f(t_n)-f(t)=\sum_{k=1}^n\left(f(t_k)-f(t_{k-1})\right)= nf(s_1)\qquad \forall\,t>\tilde{s},\;n\ge 1.
	$$
	Thus, by \eqref{limf}, we obtain
	$$
	f(\tilde{s})-f(t)=\lim_{n\to \infty}\,nf(s_1)\qquad \forall\,t>\tilde{s}.
	$$
	Therefore $f(s_1)=0$ and $f|_{(\tilde s,\infty)}\equiv a:=f(\tilde{s})$. Taking now $s,t>\tilde s$ in \eqref{rew} we get $a=0$. 
	
	Let $s'=\inf\{s\ge 0:\,f|_{(s,\infty)}\equiv 0\}$. Assume $s'>0$. Then for $s\in(0,s')$ and $t>s'$ in \eqref{rew} we have 
	$0=f(t+g(t)s)=f(s)$ - a contradiction. Thus $s'=0$. But this is impossible since $f\not\equiv 0$.
\end{proof}

\begin{lemma}\label{rcc}
	If $f$ is right continuous, then it is a continuous function.
\end{lemma}

\begin{proof} 
	It suffices to prove that $f$ is left continuous at any $s>0$. Fix arbitrary $s>0$. Then  by \eqref{rew} we have
	$$
	f(s-t)-f(s)=-f(th(s-t))\qquad \forall\,t\in[0,s],
	$$
	where $h=1/g$. By Lemma \ref{rch} we have $0<th(s-t)\le t$ for $t\in(0,s]$. Therefore, since $f$ is right continuous at $0$, for $t\to 0^+$ the left hand side turns to zero and the result follows.  
	
\end{proof}

Combining Lemmas \ref{rcrc} and \ref{rcc} we get the following result.

\begin{proposition}\label{ccc}
	If $f$ is  right continuous at some point, then it is a continuous function. 
\end{proposition}

If $f$ is monotone, then it has a countable set of points of discontinuity, i.e. in view of Proposition \ref{ccc} any monotone $f$ satisfying \eqref{rew} is continuous. It appears that this implication can be reversed with the help of right upper and lower Dini derivatives $D^+$ and $D_+$. 

\begin{lemma}\label{cvm}
	If $f$ is continuous, then it is a monotone function.
\end{lemma}
\begin{proof} 
There exists a sequence $(a_n)_{n\ge 1}$  in $(0,\infty)$ such that $\lim_{n\to\infty}\,a_n=0$ and, see \eqref{f0}, 
$$
\lim_{n\to\infty}\,\tfrac{f(a_n)}{a_n}= D^+f(0).
$$
Then, by \eqref{rew}, for every $s\ge 0$ we get
$$
D^+f(s)\ge \lim_{n\to\infty}\,\tfrac{f\left(s+a_ng(s)\right)-f(s)}{a_ng(s)}=\lim_{n\to\infty}\,\tfrac{f\left(a_n\right)}{a_ng(s)}=\tfrac{D^+f(0)}{g(s)}.
$$
Similarly, there exists a sequence $(b_n)_{n\ge 1}$  in $(0,\infty)$ such that $\lim_{n\to\infty}\,b_n=0$ and 
$$
D_+f(s)\le \lim_{n\to\infty}\,\tfrac{f\left(s+b_ng(s)\right)-f(s)}{b_ng(s)}=\lim_{n\to\infty}\,\tfrac{f\left(b_n\right)}{b_ng(s)}=\tfrac{D_+f(0)}{g(s)}.
$$
Therefore  
$$D_+f(s)\le \tfrac{D_+f(0)}{g(s)}\le \tfrac{D^+f(0)}{g(s)}\le D^+f(s),\qquad s\ge 0.$$ 
Consequently either $D^+f(s)\ge 0$ for every $s\ge 0$ or  $D_+f(s)\le 0$ for every $s\ge 0$. From \cite{SL}, Theorem 7.4.13 and its Corollary,  it follows that $f$ is either a non-decreasing or a non-increasing function.
\end{proof}

Finally we connect monotonicity of $f$ with its injectivity.
\begin{lemma}\label{mi}
	If $f$ is monotone, then it is an injective function.
\end{lemma}
		\begin{proof} 
			Suppose $f$ is not injective. Then $f(t_1)=f(t_2)$ for some $t_1,t_2\ge 0$ such that $t_1<t_2$. By \eqref{rew}, with arguments $t_1$ and $s_1=\tfrac{t_2-t_1}{g(t_1)}>0$ it follows that
			$$
			f(s_1)=f(t_1+g(t_1)s_1)-f(t_1)=f(t_2)-f(t_1)=0.
			$$  
			Since $f$ is monotone and $f(0)=0$ it follows that $f|_{[0,s_1]}\equiv 0$. Therefore $$s'=\sup\{r\ge 0:\,f|_{[0,r]}\equiv 0\}>0.$$

		      For $s,t\in[0,s')$ equation \eqref{rew} implies $f(s+g(s)t)=0$, whence $s+g(s)t\le s'$. It means that 
			$g(s)\le \tfrac{s'-s}{t}$ 	for every $s,t\in[0,s')$. Thus taking $t\in(s'-s,s')$ we get $g(s)<1$, $s\in(0,s')$. 
			
			Since $f$ is monotone then it is right continuous at a point in $[0,\infty)$ and, in view of Lemma's \ref{rcrc} and \ref{rch}, $g(s)\ge 1$ for every $s\ge 0$, a contradiction.
		\end{proof}

		\begin{proof}[Proof of Theorem \ref{oner}]
		By Proposition \ref{ccc}  and Lemma \ref{cvm} function $f$ is monotone and it follows from Lemma \ref{mi} that $f$ is   injective. Thus by referring to Theorem \ref{inj} we conclude the proof since an additive and monotone function is linear - see e.g. \cite{JA}, Ch. 2.1.1.
		  \end{proof}
		
	Theorem \ref{oner} can be extended to $f$ assuming values in a real topological vector space $X$ with dual $X^*$ which separates points on $X$, i.e. for every ${\bf x}\in X\setminus\{\bf 0\}$ there exists an $x^*\in X^*$ such that $x^*{\bf x}\neq 0$. Consequently, for ${\bf a},\,{\bf b}\in X$ if $x^*{\bf a}=x^*{\bf  b}$ for every $x^*\in X^*$, then $\bf a=\bf  b$. Note that the dual of a locally convex topological vector space $X$ separates points on $X$ (see e.g. \cite{WR}, Chapter 3: Corollary to Theorem 3.4; consult also Exercise 5(d) in the same chapter).

	\begin{corollary}\label{corx}
		Assume that $f:[0,\infty)\to X$ is non-zero, for every $x^*\in X^*$ the function $x^*\circ f$ is right continuous at some point and  $g:[0,\infty)\to[0,\infty)$.
		
		Then $(f,g)$ solves \eqref{rew}  if and only if 
		\begin{enumerate}
			\item either
			\bel{fg5}
			f(s)={\bf a}s,\quad s\in[0,\infty),\qquad\mbox{and}\qquad  g\equiv 1,
			\ee
			where ${\bf a}\in X\setminus\{\bf 0\}$,
			\item or
			\bel{fg6}
			f(s)={\bf a}\log(\alpha s+1)\quad\mbox{and}\quad g(s)=\alpha s+1,\quad s\in[0,\infty),
			\ee
			where $\alpha>0$ and ${\bf a}\in X\setminus\{\bf 0\}$.
		\end{enumerate}
	\end{corollary}
	\begin{proof} 
		Note that for any $x^*\in X^*$  the pair  $(x^*\circ f,\,g)$ solves \eqref{rew}. Moreover, since $f$ is non-zero, $x^*\circ f$ is non-zero for some $x^*\in X^*$ and it follows from Theorem \ref{oner} that 
		$$
		g(s)=\alpha s+1,\quad s\in[0,\infty),
		$$  
		where $\alpha\ge 0$. 
		
		If $\alpha=0$, then $f$ is additive and for every $x^*\in X^*$ the additive and right continuous at a point function $x^*\circ f$ has the form
		$$
		x^*f(s)=(x^*f(1))\,s,\quad s\in[0,\infty),
		$$
		whence $f(s)=f(1)s$ for $s\in[0,\infty)$, and we have \eqref{fg5} with ${\bf a}=f(1)\in X\setminus\{\bf 0\}$.
		
		If $\alpha>0$, then by Theorem \ref{oner} for every $x^*\in X^*$ either $x^*\circ f\equiv 0$, or 
		$$
		x^*f(s)=a\log(\alpha s+1),\quad s\in[0,\infty),
		$$ 
		where $0\neq a\in\R$; in the second case
		$$
		a=x^*f\left(\tfrac{e-1}{\alpha}\right).
		$$
		
		Consequently, for every $x^*\in X^*$ in both cases we have
		$$
		x^*f(s)=x^*f\left(\tfrac{e-1}{\alpha}\right)\,\log(\alpha s+1),\quad s\in[0,\infty),
		$$
		i.e.
		$$
		f(s)=f\left(\tfrac{e-1}{\alpha}\right)\,\log(\alpha s+1),\quad s\in[0,\infty).
		$$
		Thus we get \eqref{fg6} with ${\bf a}=f\left(\tfrac{e-1}{\alpha}\right)\in X\setminus\{\bf 0\}.$
	\end{proof}

	\section{CGS equations on a vector space}
		Throughout this section  $U=V$, a vector space over a field $\mathbb F$. As it has already been observed, see Remark \ref{rem2}, if $(f,h)$ solves \eqref{equ} on $V$, then $\mathrm{Ker}(h)=\emptyset$. Therefore, due to Remark \ref{rem3}, equations \eqref{equ} and \eqref{rew} are equivalent with $hg\equiv 1$. 
		
		We assume that $(M,+)$ is a unital magma with the two-sided cancelation property.
		
		\begin{theorem}\label{injR}
			Assume that $f:V\to M$ is a non-zero function and $g:V\to \mathbb F$. 
			
			Then $(f,g)$ solves \eqref{rew} if and only if $f$ is additive and $g\equiv 1$.		
		\end{theorem}
		
		\begin{proof}
			First, we prove that $f$ is additive, i.e. 
			\bel{Cauc}f(s+t)=f(s)+f(t)\ee
			for all $s,t\in V$.
			
			To this end we observe that for $s\in V$, $f(s)\neq \mathbf 0$ implies $g(s)=1$. Indeed, if  $g(s)\neq 1$, then  $s+g(s)t=t$ for $t=\tfrac{s}{1-g(s)}$. Thus \eqref{rew} for such $s$ and $t$  yields $f(s)=\mathbf 0$, a contradiction.  Consequently,  \eqref{Cauc} holds true for $s,t\in V$ such that at least one of $f(s)$ and $f(t)$ is not zero. To see this fact consider separately the cases: (a) both $f(s)$ and $f(t)$ are non-zero, (b) exactly one of $f(s)$ and $f(t)$ is non-zero. In the latter case use the identity $\mathbf a+\mathbf 0=\mathbf 0+\mathbf a=\mathbf a$, which holds for every $\mathbf a\in M$.
			
		It suffices to prove \eqref{Cauc} for $s,t\in V$ such that $f(s)=f(t)=\mathbf 0$. Assume \eqref{Cauc} does not hold for such $s,t\in V$, i.e. $f(s+t)\neq \mathbf 0$. Then $g(s+t)=1$ and \eqref{rew} implies
			\bel{mat0}
			f(t)=f(s+t-s)=f(s+t)+f(-s).
			\ee
			Note that $f(-s)=\mathbf 0$. Otherwise $g(-s)=1$ and \eqref{rew} yields $\mathbf 0=f(-s+s)=f(-s)+f(s)$, a contradiction. Therefore, by \eqref{mat0} we get $f(t)=f(s+t)\neq \mathbf 0$, a contradiction.  
			
		Second, we prove that $g\equiv 1$. Assume not, i.e. $g(s)\neq 1$ for some $s\in V$. For arbitrary $u\in V$ set $t:=\tfrac{u}{g(s)-1}\in V$. Then, by \eqref{Cauc}, we have
			\bel{efu}
			f(u)=f((g(s)-1)t)=f(s+g(s)t-s-t)=f(s+g(s)t)+f(-s-t).
			\ee
				But \eqref{rew} and \eqref{Cauc} yield $f(s+g(s)t)=f(s)+f(t)=f(s+t)$ and $f(s+t)+f(-(s+t))=\mathbf 0$. Consequently,  \eqref{efu} implies $f\equiv \mathbf 0$, a contradiction. 
		\end{proof}

\vspace{3mm} {\bf Acknowledgement.} KB research was supported by the Institute of Mathemtics of the University of Silesia (Iterative Functional Equations and Real Analysis program). JW research was supported in part by Grant 2016/21/ B/ST1/00005 of the National Science Center, Poland

		\end{document}